\documentclass[a4paper]{amsart}

\usepackage{hyperref}
\hypersetup{
    colorlinks=true,       
    linkcolor=blue,          
    citecolor=blue,        
    filecolor=blue,      
    urlcolor=blue           
}
\usepackage[all]{xy}
\usepackage{tikz,graphicx}
\usepackage{amssymb}
\usepackage{booktabs}
\usepackage[centertags]{amsmath}
\usepackage{amsfonts}
\usepackage{enumerate}
\newcommand{\Gal}{{\mathrm {Gal}}}

\newcommand{\Aut}{\mathrm{Aut}}
\newcommand{\PGL}{{\mathrm{PGL}}}

\newcommand{\Z}{{\mathbb Z}}
\newcommand{\Q}{{\mathbb Q}}

\newcommand{\F}{{\mathbb F}}

\newcommand{\bq }{\begin{equation}}
\newcommand{\eq }{\end{equation}}
\theoremstyle{plain}
\newtheorem{thm}{Theorem}[section]
\newtheorem{lem}[thm]{Lemma}
\newtheorem{lema}[thm]{Lemma}
\newtheorem{prop}[thm]{Proposition}

\newtheorem{conj}[thm]{Conjecture}

\newtheorem{cor}[thm]{Corollary}
\newtheorem{rem}[thm]{Remark}

\theoremstyle{definition}
\newtheorem{defn}[thm]{Definition}

\theoremstyle{example}

\title{Bielliptic smooth plane curves and quadratic points}

\author[E. Badr] {Eslam Badr}
\address{$\bullet$\,\,Eslam Badr}
\address{Department of Mathematics,
Faculty of Science, Cairo University, Giza-Egypt}
\email{eslam@sci.cu.edu.eg}

\author[F. Bars] {Francesc Bars}
\address{$\bullet$\,\,Francesc Bars}
\address{Departament Matem\`atiques, Edif. C, Universitat Aut\`onoma de Barcelona\\
08193 Bellaterra, Catalonia} \email{francesc@mat.uab.cat}

\thanks{F. Bars is partially supported by MTM2016-75980-P}


\begin{document}
\maketitle

\begin{abstract} Let $C_k$ be a smooth projective curve over a global field $k$, which is neither rational nor elliptic. Harris-Silverman \cite{HaSi},
when $p=0$, and Schweizer, when $p>0$ together with an extra condition on
the Jacobian variety $\operatorname{Jac}(C_k)$ arising from Mordell's conjecture, showed that $C$ has infinitely many
quadratic points over some finite field extension $L/k$ inside
$\overline{k}$ (a fixed algebraic closure of $k$) if and only if $C$
is hyperelliptic or bielliptic.

Now, let $C_k$ be a smooth plane curve of a fixed degree $d\geq4$
with $p=0$ or $p>(d-1)(d-2)+1$ (up to an extra condition on $\operatorname{Jac}(C_k)$ in positive characteristic). Then, $C_k$
admits always finitely many quadratic points unless $d=4$; see
Theorem \ref{thm3.3}. A so-called \emph{geometrically complete
families} for the different strata of smooth bielliptic plane
quartic curves by their automorphism groups, are given; see Theorem
\ref{biellipticquartics}. Interestingly, we show (in a very simple
way) that there are only finitely many quadratic extensions
$k(\sqrt{D})$ of a fixed number field $k$, in which we may have more
solutions to the Fermat's and the Klein's equations of degree
$d\geq5$; $X^d+Y^d-Z^d=0$ and $X^{d-1}Y+Y^{d-1}Z+Z^{d-1}X=0$
respectively, than these over $k$ (the same holds for any non-singular projective plane equation of degree $d\geq 5$ over $k$, and also in general when $k$ is a global
field after imposing an extra condition on $\operatorname{Jac}(C_k)$); see Corollary
\ref{FermatKlein}.

Finally, given a stratum $\mathcal{M}(G)$ of smooth plane bielliptic
quartic curves over a number field $k$ associated to an automorphism
group $G$, we conjecture in section \S 3 that there are subsets
$\mathcal{E}, \mathcal{D}\subset\mathcal{M}(G)$ of infinite
cardinality, such that all members of $\mathcal{E}$ (resp.
$\mathcal{D}$) have finitely (resp. infinitely) many quadratic
points over $k$. We support our claim when $k=\Q$ and $G=\Z/6\Z,$ or
$\operatorname{GAP}(16,13)$; see Theorems \ref{thm1}, \ref{thm2},
\ref{thm3} and \ref{thm4}.
\end{abstract}

\begin{section}{The interplay between hyperelliptic (resp. bielliptic) curves and quadratic points}
Let $\overline{k}$ (resp. $k^{\operatorname{sep}}$) be a fixed algebraic (resp. separable) closure of a field $k$ of
characteristic $p\neq2$. By $C_k$ we mean a smooth projective curve defined over $k$
of geometric genus $g_C\geq2$ (that is, the genus of the base extension
$C_{\overline{k}}:=C\otimes_k\overline{k}$ is at least two), and
non-trivial automorphism group $\Aut(C_{\overline{k}})$. The set of
all $k$-points on $C_k$ is denoted by $C(k)$.

An arithmetic geometer finds a lot of interest to investigate the
cardinality of $C(k)$. When $k$ is a global field; i.e. when $k$ is
a finite field extension of either $\Q$ or $\F_p(T)$, or it is the
function field of $\mathbb{P}^1$ over the finite field $\F_p$ {(in
this case, we denote the finite field $k\cap\overline{\F}_p$ by
$\F_q$, where $q$ is a power of $p$, a priori)}.

In zero characteristic, we have the following result on \emph{Mordell's Conjecture} due to
Faltings \cite{Fa1, Fa2}:
\begin{thm}[Faltings] Given a smooth projective curve $C_k$ as above defined over a number field $k$, the set $C(k)$ is always finite.
\end{thm}
{On the other hand, we obtain by Grauert \cite{Gra} and Samuel
\cite[Theorem 4 and 5b]{Sam} the next result in positive
characteristic:}
 \begin{thm}[Grauert-Samuel]\label{GrauertSamuel} Let $C_k$ be a smooth projective curve over a global field $k$ of characteristic $p>0$. Assume also that $C_k$ is
 conservative (see the definition in \S \ref{conservativedefn}). Then, $C(k)$ is always
 finite except {possibly} when $C_k\otimes_k k^{\operatorname{sep}}$ is isomorphic to a smooth
 projective curve $C'$ over a finite field $\mathbb{F}_{q^n}$ (in this situation, $C_k$ or $C_k\otimes_k k^{\operatorname{sep}}$ is called an isotrivial curve).
 More concretely, for $C(k)$ to be infinite, it suffices the existence of a finite Galois extension $\ell'/k$, an $\ell'$-isomorphism $\varphi:C_{\ell'}=C_k\otimes_k\ell'\rightarrow C'\otimes_{\mathbb{F}_{q^n}}\ell'$,
 an injection $\operatorname{Gal}(\ell'/k)\hookrightarrow\operatorname{Aut}(C'\otimes_{\mathbb{F}_{q^n}}\ell'):\,s\mapsto j_s:=\varphi^s\circ
\varphi^{-1}$, and a point $z\in C'(\ell')\setminus
C'(\overline{\F}_p)$ satisfying $j_s(z)=z^s$ for all $s\in\operatorname{Gal}(\ell'/k)$. Under these conditions, it exists a finite
family $(x_i)_{i\in I}$ of points of $C'(\ell')$ with $x_i^s=j_s(x_i)$
for all $i\in I,\,s\in\operatorname{Gal}(\ell'/k)$, and moreover the infinite set
$C(k)$ is given by:
$$C(k)=(\bigcup_{i\in I,{m\geq 0}}\,\varphi^{-1}(f^{m}(x_i)))\cup\,\left(C(k)\cap(\varphi^{-1}(C'(\overline{\F}_{{q^n}})))\right)$$
{where $f:x\mapsto x^{q^{n\ss_0}}$ and $\ss_0$ is the strict least
positive integer $\ss$ satisfying $(x\mapsto x^{q^{n\ss}})\circ
j_s=j_s\circ(x\mapsto x^{q^{n\ss}})$ for all $s\in
\operatorname{Gal}(\ell'/k)$}.\footnote{{It is assumed in
\cite[Theorem 5b]{Sam} that all automorphisms of $C'$ are also
defined over
$\F_{q^n}$, in particular, $f\circ j_s=j_s\circ f$ for some power of the Frobenius. Accordingly, we will impose the latter condition directly.}}
%
 \end{thm}

\begin{rem}
Grauert-Samuel Theorem requires $C_k$ to be geometrically non-singular instead of being conservative, but if $C_k$ is
conservative then it is also geometrically non-singular by \cite{Gra}.

\end{rem}
\begin{rem}
An example of Theorem \ref{GrauertSamuel}, where an infinite number
of points occur, is the Fermat curve $C_k:X^d+Y^d=Z^d$ over the
global field $k=\F_{p}(w,v)/(w^d+v^d-1)$ with $d$ and
$p$ coprime. It is isomorphic to the Fermat curve over $\F_p$, and
we get an infinite set of points over $k$ on it namely,
$(f^n(w),f^n(v),1)$ for each positive integer $n$ by using the point
$(w,v,1)\in C(k)\setminus C(\overline{\F}_p)$ where $f$ is the
Frobenious $x\mapsto x^p$.
\end{rem}

For a finite field extension $L/k$ inside $\overline{k}$, the set of quadratic points of $C_k$ over $L$, denoted by $\Gamma_2(C,L)$, is given by
$$\Gamma_2(C,L):=
\bigcup\big\{C(L')\,:\,L\subseteq
L'\subseteq\overline{k}\,\text{with}\,[L':L]\leq 2\big\},$$ where (by an abuse of notation) $C(L')$ denotes the set of $L'$-points
of $C_{L'}:=C\otimes_kL'$.

It is a natural question to study whether
$\Gamma_2(C,L)$ is finite or not.

\begin{defn} A smooth projective curve $C_k$ is called \emph{hyperelliptic} (resp. \emph{bielliptic}) \emph{over $k$} if there exists a degree two $k$-morphism to a projective line
 $\mathbb{P}^1_k$ (resp. to an elliptic curve $E_k$) over $k$. We simply call it \emph{hyperelliptic} (resp. \emph{bielliptic}) if $C_{\overline{k}}$ is hyperelliptic (resp. bielliptic) over $\overline{k}$.
\end{defn}
\begin{rem}
Clearly, if $C_k$ is hyperelliptic over $k$ and $k$ is not a finite field, then $\Gamma_2(C,k)$ is an infinite set. Also, if $C_k$ is bielliptic over $k$ and $E_k$ has infinitely many $k$-points, then
$\Gamma_2(C,k)$ is again an infinite set.
 \end{rem}

The following result is well-known in the literature (cf. \cite{Sch} for $(ii)$).
\begin{prop}\label{prop2.3} Let $C_k$ be a smooth projective curve over $k$. Then,
\begin{enumerate}[(i)]
\item $C_k$ is hyperelliptic if and only if there exists a (hyperelliptic) involution $w\in
\Aut(C_{\overline{k}})$, having exactly $2g_C+2$ fixed points. In particular, if $C_k$ is hyperelliptic, then $w$ is unique, defined over a finite purely inseparable extension $\ell/k$ of
$k$, and it is called the \emph{hyperelliptic involution} of $C_k$.
\item $C_k$ is bielliptic if and only if there exists a (bielliptic) involution $\tilde{w}\in
\Aut(C_{\overline{k}})$, having $2g_C-2$ fixed points. If $C_k$ is
bielliptic and $g_C\geq 6$, then there is a unique
bielliptic involution, which belongs to the center of
$\Aut(C_{\overline{k}})$ and defined over a finite
purely inseparable extension $\ell$ of $k$.
\end{enumerate}
\end{prop}
%
\subsection{Conservative curves}\label{conservativedefn} Let $C_k$ be a smooth projective curve over a global field $k$ of characteristic $p>0$. The genus of $C_k$ relative to $k$ is defined to be the integer $g_{C,k}$ that makes the Riemann-Roch formula hold, that is, for any $k$-divisor $D$ of
$C$, of sufficiently large degree; $\ell(D)=\operatorname{deg}(D)+1-g_{C,k}$, where $\ell(D)$ is the dimension
of the $k$-(Riemann-Roch) vector space associated to $D$. The relative genus may change under inseparable extensions of $k$ inside $\overline{k}$, see for example \cite{Tate}. The absolute genus of $C_k$ is defined to be
the genus of $C_{\overline{k}}$ relative to $\overline{k}$, in particular it equals to the geometric genus $g_C$ we have seen before.

The relative genus $g_{C,k}$ to $k$ is an upper bound for the absolute genus $g_C$. We call $C_k$ \emph{conservative over $k$}, if $g_C=g_{C,k}$ (in particular, it is not genus-changing under inseparable extensions between
$k$ and $\overline{k}$).

First, we prove:
\begin{prop} \label{lem1}\label{Lem2.4} Let $C_k$ be a smooth projective curve defined over a global field $k$ of characteristic $p>0$, that is conservative over $k$. Assume also that $C_k$ is hyperelliptic with hyperelliptic
involution $w$ defined over a finite purely inseparable extension $\ell/k$ in $\overline{k}$. Then, there is a (unique) degree two $\ell$-morphism $\varphi$ to a conic $Q$ over $\ell$. Moreover, if $C_{\ell}$ (or more generally if $C_{\ell}/\langle w\rangle$) has an $\ell$-point, then we reduce to that $Q$ is $\ell$-isomorphic to $\mathbb{P}^1_{\ell}$.
\end{prop}
\begin{proof} By assumption $C_{\ell}/\langle w\rangle$ is a genus $0$ curve defined over $\ell$. Thus it corresponds to a conic over $\ell$, since by definition  it is a twist for $\mathbb{P}^1_{\ell}$. Next, the covering $\pi:C_{\ell}\rightarrow C_{\ell}/\langle w\rangle$ is Galois (being cyclic of degree $2$), hence $\pi$ is defined over a separable extension of $\ell$ inside $\overline{k}$, a priori. Fix a separable closure $\ell^{\operatorname{sep}}\subseteq\overline{k}$ of $\ell$ and let $\operatorname{Gal}(\ell^{\operatorname{sep}}/\ell)$ denotes the absolute Galois group. By the uniqueness of the hyperelliptic involution $w$, $\pi$ and $^{\sigma}\pi$ only differs by an automorphism $\xi_{\sigma}$ of $\mathbb{P}^1_{\ell^{\operatorname{sep}}}$, for any $\sigma\in
\Gal(\ell^{\operatorname{sep}}/\ell)$. In other words, for any $\sigma\in\Gal(\ell^{\operatorname{sep}}/\ell)$, we obtain $\xi_{\sigma}\in {\PGL}_2(\ell^{\operatorname{sep}})$; the projective general
linear group of $2\times2$ matrices over $\ell^{\operatorname{sep}}$), where $^{\sigma}\pi=\xi_{\sigma}\circ\pi$. It can be easily checked that $\xi_{\sigma\tau}=\,^{\sigma}\xi_{\tau}\circ\xi_{\sigma}$ for all $\sigma,\tau\in \Gal(\ell^{\operatorname{sep}}/\ell)$. Therefore, $$\xi:\Gal(\ell^{\operatorname{sep}}/\ell)\rightarrow{\PGL}_2(\ell^{\operatorname{sep}}):\sigma\mapsto\xi_{\sigma}$$ defines a 1-cocycle, in particular, an
element of the first Galois cohomology set $\operatorname{H}^1(\Gal(\ell^{\operatorname{sep}}/\ell),\PGL_2(\ell^{\operatorname{sep}}))$. Using the Twisting Theory for Varieties (cf. \cite[III.1]{Se}), it exists a conic $Q$ (a twist for $\mathbb{P}^1_{\ell}$) over $\ell$ and an isomorphism $\varphi_0:Q\rightarrow C_{\ell}/\langle w\rangle$ given by the rule $\xi_{\sigma}=\,^{\sigma}\varphi_0\circ\varphi_0^{-1}$, for all
$\sigma\in \Gal(\ell^{\operatorname{sep}}/\ell)$. Consequently, $\varphi:=\varphi_0^{-1}\circ\pi:C_{\ell}\rightarrow Q$ is an $\ell$-morphism from $C_{\ell}$ to $Q$.

The rest is direct, since a conic over $\ell$ that has an
$\ell$-point is $\ell$-isomorphic to $\mathbb{P}^1_{\ell}$. This obviously happens if $C_{\ell}$ or
$C_{\ell}/\langle w\rangle$ has an $\ell$-point via the morphism
$\varphi$.
\end{proof}

\begin{cor} Let $C_k$ be a smooth projective curve defined over a global field $k$ of characteristic $p>0$, that is conservative over $k$. Then, $C_k$ is
hyperelliptic if and only if there exists a finite extension $L/k$ inside $\overline{k}$ where $C_k\otimes_kL$
is hyperelliptic over $L$. In this situation, $\Gamma_2(C,L)$ is an
infinite set.

\end{cor}

Also, we show:
\begin{prop}\label{100} Let $C_k$ be a smooth projective curve defined over a global field $k$ of characteristic $p>0$, that is conservative over $k$.
Then, $C_k$ is bielliptic if and only if there exists a finite
extension $L/k$ inside $\overline{k}$ where $C_k\otimes_kL$ is bielliptic over $L$ (hence, $\Gamma_2(C,L')$ is
an infinite set for some finite extension $L'/L$ inside $\overline{k}$).
\end{prop}
\begin{proof} Assume that $C_k$ is bielliptic and consider a bielliptic involution $\tilde{w}\in
\Aut(C_{\overline{k}})$ as in Proposition \ref{prop2.3}. Since
$g_C\geq2$, $\Aut(C_{\overline{k}})$ is a finite group, and so we
only have finitely many possibilities for the Galois group
$\Gal(\overline{k}/k)$-action on $\tilde{w}$. Accordingly,
$\tilde{w}$ must be defined over a finite field extension $L_0/k$ inside $\overline{k}$. Because $C_k$ is also conservative
over $k$, then by making a finite extension $L/L_0$ with
$L\subseteq\overline{k}$, we get a degree two $L$-morphism from
$C_k\otimes_kL$ to a genus one curve that has $L$-points, hence to an
elliptic curve $E$ over $L$, and hence $C_k$ is bielliptic over $L$.
Finally, it suffices to apply base extension to some $L\subseteq
L'\subseteq\overline{k}$ of finite index so that $E\otimes_LL'$ has
positive rank. Consequently, $\Gamma_2(C,L')$ is infinite and we
conclude.
\end{proof}

Furthermore, inspired by the case of number fields (Theorems
\ref{AbramovichHarrisSilvermanI} and \ref{AbramovichHarrisSilvermanII} below), we have \cite[Theorem 5.1]{Sch}:
\begin{thm}[Schweizer, version I]\label{SchweizerI} Let $C_k$ be a conservative smooth projective curve over a finite field extension $k$ of $\F_q(T)$; i.e. $k$ is a global field of characteristic $p>0$. Under the conditions that; $g_C\geq3$, $C(k)\neq\emptyset$, $C_k$ is not hyperelliptic over $k$, and that the
Jacobian variety $\operatorname{Jac}(C_{k})$ over $\overline{k}$ has
no non-zero homomorphic images defined over $\overline{\F}_q$, then
$\Gamma_2(C,k)$ is an infinite set only if
$\operatorname{Jac}(C_{k})$ over $k$ contains an elliptic curve
$E_k$ of positive rank, moreover there is a degree two morphism from
$C_k$ to $E_k$.
\end{thm}

\begin{rem} The conditions on $\operatorname{Jac}(C_k)$ imply that $E_k$ is not isotrivial,
(i.e. its $j$-invariant of $E_k$ does not belong to
$\overline{\F}_q$) and also that any factor of the Jacobian is a
Jacobian of an isotrivial curve {(as it happens, for example, by the
Jacobian of the Fermat curve $X^d+Y^d=Z^d$ over a global field of
positive characteristic, relatively prime with $d$, which is
isotrivial)}.
\end{rem}
\begin{rem}\label{rem1.11} To ensure that the degree two morphism from $C_k$ to
$E_k$ is also defined over $k$, it suffices to assume that
$g_C\geq6$ and then to follow the argument of \cite[Lemma 5]{HaSi},
provided that Castellnuovo's inequality holds over {the function
field extensions over a global field $k$ that are involved in
Castellnuovo's inequality (we mention that by the proof of
\cite[Lemma 5]{HaSi} all of these extensions are of degree $2$
in our situation)}. By \cite[Theorem 3. 11. 3]{Sti},
the inequality is true when $k$ is perfect, which is not the case.
However, in the proof of \cite[Theorem 3.11.3]{Sti} under the
assumption that $C_k$ is conservative, one only needs
to take the characteristic $p$ big enough so that inseparable
extensions do not appear (for example, to be greater than the
degrees of the function field extensions involved in Castellnuovo's
inequality). Consequently, under the hypothesis of Theorem
\ref{SchweizerI} {together with the assumptions} $g_C\geq 6$ and
$p>2$, we have $C_k$ is bielliptic over $k$.
\end{rem}

The next result is a weaker version of Theorem \ref{SchweizerI}.
\begin{thm}[Schweizer, version II]\label{SchweizerII}
Let $C_k$ be a smooth projective curve defined over a global field
$k$ of characteristic $p>0$, that is conservative over $k$. Under
the conditions that; $g_C\geq3$, and that the Jacobian variety
$\operatorname{Jac}(C_{k})$ over $\overline{k}$ has no non-zero
homomorphic images defined over $\F_q$, then, there exists a finite
field extension $L/k$ inside $\overline{k}$ such that
$\Gamma_2(C,L)$ is infinite if and only if $C_k$ is bielliptic or
hyperelliptic.
\end{thm}

\subsection{Number fields}
Let us now assume that $k$ is a global field of characteristic zero,
that is, $k$ is a number field. If we follow the same line of discussion in
the proofs of Propositions
\ref{prop2.3}, \ref{Lem2.4} and \ref{100}, we deduce,
since inseparable extensions between $k$ and $\overline{k}$ do not
exist, that:
\begin{prop}
A smooth projective curve $C_k$ defined over a number field $k$ is
hyperelliptic with hyperelliptic involution $w$ such that $C/\langle w\rangle(k)\neq\emptyset$ if and only if it is hyperelliptic over
$k$. Also, for $g_C\geq6$, $C_k$ is bielliptic if and only if it is
bielliptic over $k$.
\end{prop}
The next results are quite known to the specialists (they follow
from the arguments of Abramovich-Harris in \cite{abha}, or from
Harris-Silverman in \cite{HaSi}). One also can read a proof in
\cite{Ba1}.

%
%
%
%
%
%

 \begin{thm}\label{AbramovichHarrisSilvermanI} \mbox{}
Let $C_k$ be a smooth projective curve over a number field $k$ with $g_C\geq 2$. Hence, $\Gamma_2(C,k)$ is an infinite set if and only if $C_k$ is
hyperelliptic over $k$ or bielliptic over $k$, such that it exists a degree two $k$-morphism form $C_k$ to an elliptic curve $E_k$ of positive rank over $k$.
\end{thm}

The next result in \cite{HaSi} is a weaker version of Theorem
\ref{AbramovichHarrisSilvermanI}, not {controlling} the base field
for quadratic points:
\begin{thm}[Harris-Silverman]\label{AbramovichHarrisSilvermanII}
Let $C_k$ be a smooth projective curve over a
number field $k$ with $g_C\geq 2$. Hence, there exists a finite field extension $L/k$ inside $\overline{k}$ such that
$\Gamma_2(C,L)$ is an infinite set if and only if $C_k$ is
hyperelliptic or bielliptic.
\end{thm}

\end{section}
\begin{section}{Bielliptic smooth plane curves and their quadratic points}
Let $k$ be a field and $k\subseteq
L\subseteq\overline{k}$ be a field extension. A curve $C$ is said to
be a \emph{smooth $(k,L)$-plane curve of degree $d$} or
equivalently, $L$ is a \emph{plane model field of definition for
$C$}, if $C$ as a smooth projective curve is defined over $k$ such
that $C_L:=C\otimes_kL$ admits a non-singular plane model over $L$,
that is, $C_L$ is $L$-isomorphic to a non-singular homogenous projective
polynomial equation $F(X,Y,Z)=0$ of degree $d$ with coefficients in
$L$. When $\overline{k}$ is a plane model field of definition for
$C$, then $C_{\overline{k}}$ has a unique $g^2_d$-linear series
modulo $\operatorname{PGL}_3(\overline{k})$-conjugation, which
allows us to embed $C_{\overline{k}}$ into
$\mathbb{P}^2_{\overline{k}}$ as a non-singular plane model over
$\overline{k}$ of degree $d$ for some $d$ (in this case,
$g:=g_C=(d-1)(d-2)/2$). Also, we call $C$ a \emph{smooth plane curve
of degree $d$ over $k$} when $k$ itself is plane model field of definition
for $C$.

The present authors and et al. addressed in \cite[\S 2]{BBL} the problem where non-singular plane models of smooth projective curves (also of their twists) are defined. For instance, we showed:
\begin{thm}[Badr-Bars-Garc\'{\i}a]\label{thmfields}
Given a smooth $(k,k^{\operatorname{sep}})$-plane curve $C$ of degree $d\geq4$, it
does not necessarily have a non-singular plane model defined over
the field $k$. However, it does in any of the following cases; if $d$ is coprime with $3$, if $C(k)\neq\emptyset$, or if the $3$-torsion $\operatorname{Br}(k)[3]$ of the Brauer group $\operatorname{Br}(k)$ is trivial.
In general, $C$ is a smooth $(k,L)$-plane curve of degree $d$ for some cubic Galois field extension $L/k$. \end{thm}

It is a basic fact in algebraic geometry that a smooth
$(k,\overline{k})$-plane curve of degree $d\geq4$ is
non-hyperelliptic. On the other hand, the (geometric) gonality of a
smooth projective curve $C$ is defined to be the minimum degree of a
$\overline{k}$-morphism from $C_{\overline{k}}$ to the projective
line $\mathbb{P}^1_{\overline{k}}$. For the special case of smooth
$(k,\overline{k})$-plane curves $C$ of degree $d\geq4$, the gonality
equals to $d-1$ (cf. Namba \cite{Na} for zero characteristic and
Homma \cite{Ho} for positive characteristic, also one can read the
assertion in \cite[p. 341]{Yoshihara}). Consequently, when $d\geq6$,
the gonality of $C$ is at least $5$, and $C$ can
not be bielliptic. When $d=5$, we use the fact that
$C_{\overline{k}}$ has an automorphism $\tilde{w}$ of order $2$ (a
bielliptic involution) that also fixes $2g-2=10$ points on
$C_{\overline{k}}$. Following the techniques in \cite{BaBa2016IA}, we may assume (up to
$\operatorname{PGL}_3(\overline{k})$-conjugation) that
$\tilde{w}:X\mapsto X,\,Y\mapsto Y, Z\mapsto-Z$ and
$C_{\overline{k}}:Z^4L_{1,Z}+Z^2L_{2,Z}+L_{5,Z}$, where $L_{i,Z}$
denotes a homogenous degree $i$, binary form in $X,Y$. Thus,
$\tilde{w}$ fixes exactly $d+1=6<10$ points on $C_{\overline{k}}$, a
contradiction, and $C$ can not be bielliptic. This proves the
result:

\begin{prop}\label{nonhypergenus}
A smooth $(k,\overline{k})$-plane curve $C$ of degree $d\geq5$ is neither hyperelliptic nor bielliptic. Also, $C$ is never hyperelliptic when $d=4$.
\end{prop}
\subsection{Smooth plane quartic curves, which are bielliptic}
Next, we characterize smooth $(k,\overline{k})$-plane quartic curves that might be bielliptic, hence, $\Gamma_2(C,L)$ is infinite for some finite field extension $L/k$ inside $\overline{k}$.

Let $\mathcal{M}_g$ be the (coarse) moduli space, representing
$\overline{k}$-isomorphism classes of smooth curves of genus $g$,
and define the substratum
$\widetilde{\mathcal{M}_g^{\operatorname{Pl}}}(G)\subset\mathcal{M}_g$,
where $G$ is a finite non-trivial group, whose $\overline{k}$-points are $\overline{k}$-isomorphism
classes of smooth plane curves $[C]$, such that $\operatorname{Aut}(C_{\overline{k}})$
is $\operatorname{PGL}_3(\overline{k})$-conjugated to $\varrho(G)$,
for some injective representation
$\varrho:G\hookrightarrow\operatorname{PGL}_3(\overline{k})$.

Lercier-Ritzenthaler-Rovetta-Sijsling in
\cite[\S2]{Ritothers} introduced the notions of complete, finite and representative families for strata of the moduli space $\mathcal{M}_g$
 when $k$ is {any field of characteristic $p=0$ or $p>2g+1$}.
 In the case of plane curves, a family $\mathcal{C}$ of smooth plane curves over $k$ is said to be \emph{complete over $k$ for $\widetilde{\mathcal{M}_g^{\operatorname{Pl}}}(G)$} if,
 for any algebraic extension $k'/k$ inside $\overline{k}$ and any $k'$-point $[C]/k'$ in the stratum $\widetilde{\mathcal{M}_g^{\operatorname{Pl}}}(G)$,
 there exists a non-singular plane model for $C$ defined over $k'$ in the family $\mathcal{C}$, moreover if such a model is always unique, then the family $\mathcal{C}$
 is \emph{representative over $k$ for $\widetilde{\mathcal{M}_g^{\operatorname{Pl}}}(G)$}. A family $\mathcal{C}$ is \emph{geometrically complete} (resp. \emph{geometrically representative})
 if $\mathcal{C}\otimes_k\overline{k}$ is complete (resp. representative) over $\overline{k}$.
%
%
%
%

\begin{thm}[Bielliptic quartic curves]\label{biellipticquartics} Let $C$ be {a smooth $(k,\overline{k})$-plane quartic curve over a field $k$} of characteristic $p=0$ or
$p>7$ \footnote{In case that $k$ is a perfect field, Theorem \ref{thmfields} allows us
to always assume that $C$ is a smooth plane quartic curve over the base field $k$, since $d=4$ is coprime with $3$.}. Then, $C$ is bielliptic if and only if
$C\otimes_k\overline{k}$ is isomorphic to a non-singular plane model of
the form $Z^4+Z^2L_{2,Z}+L_{4,Z}=0$, where $L_{i,Z}$ is a homogenous
binary form in $\overline{k}[X,Y]$ of degree $i$ (in this case,
$\tilde{w}:X\mapsto X,\,Y\mapsto Y,\,Z\mapsto -Z$ is a bielliptic
involution).

Table \ref{Hennclassification} below gives a geometrically complete
families for each substrata (in $\operatorname{GAP}$ library
\cite{gap} notation) of $\mathcal{M}_3$ of smooth plane quartic
curves over $\overline{k}$, which are bielliptic.

\scriptsize
\begin{center}
\begin{table}[!th]
  \renewcommand{\arraystretch}{1.3}
  \caption{Bielliptic geometrically complete classification}
  \label{Hennclassification}
 \vspace{4mm} 
  \centering
\begin{tabular}{|c|c|c|}
  \hline
  $\operatorname{Aut}(C_{\overline{k}})$ & Families & Restrictions \\
                                     &       &              \\\hline\hline

  $\Z/2\Z$& $Z^4+Z^2L_{2,Z}(X,Y)+L_{4,Z}(X,Y)$ & $L_{2,Z}(X,Y)\neq0,\,not\, below$   \\\hline
  $\Z/2\Z\times \Z/2\Z$& $Z^4+Z^2(bY^2+cX^2)+(X^4+Y^4+aX^2Y^2)$& $a\neq\pm b\neq c\neq\pm a$    \\\hline
    $\Z/{6}\Z$& $Z^4+aZ^2Y^2+(X^3Y+Y^4)$ & $a\neq0$    \\\hline
$\mathbf{\operatorname{\mathbf{S}}_3}$ &
$(X^3+Y^3)Z+X^2Y^2+aXYZ^2+bZ^4$& $a\neq b,\,ab\neq0$    \\\hline
$\operatorname{D}_4$& $Z^4+bXYZ^2+(X^4+Y^4+aX^2Y^2)$&
$b\neq0,\pm\frac{2a}{\sqrt{1-a}}$    \\\hline
  $\operatorname{GAP}(16,13)$& $Z^4+(X^4+Y^4+aX^2Y^2)$& $\pm a\neq0,2,6,2\sqrt{-3}$   \\\hline
  $\operatorname{S}_4$ & $Z^4+aZ^2(Y^2+X^2)+(X^4+Y^4+aX^2Y^2)$ & $a\neq0,\frac{-1\pm\sqrt{-7}}{2}$  \\\hline
    $\operatorname{GAP}(48,33)$ & $Z^4+\left(X^4+Y^4+(4\zeta_3+2)X^2Y^2\right)$ & $-$    \\\hline
  $\operatorname{GAP}(96,64)$& $Z^4+(X^4+Y^4)$& $-$   \\\hline
$\mathbf{\operatorname{\mathbf{PSL}}_2(\mathbb{F}_7)}$&
$X^3Y+Y^3Z+Z^3X$& $-$   \\\hline
 \end{tabular}
\end{table}
\end{center}
\normalsize

The algebraic restrictions for the parameters over $\overline{k}$, in the last column,
are taken so that the defining equation is non-singular and has no
bigger automorphism group. For example, the term ``not below'' means to
assume more restrictions for no larger automorphism group to occur.
\end{thm}
\begin{rem}
The two highlighted cases in Table \ref{Hennclassification} are not
in the prescribed form $Z^4+Z^2L_{2,Z}+L_{4,Z}=0$, that is
$\operatorname{diag}(1,1,-1)$ is not a bielliptic involution.
However, they do up to
$\operatorname{PGL}_3(\overline{k})$-conjugation. The way to this is
simple; any $3\times3$ projective linear transformation of order two
of the plane fixes (pointwise) a line $\mathcal{L}\subset\mathbb{P}^2_{\overline{k}}$,
called its axis. So, we may ask for a change of variables $\phi$
such that the transformed $\mathcal{L}$ becomes the projective line $Z=0$.
For example, for $G=\operatorname{S}_3$, we may apply the change of
variables
$$\left(
   \begin{array}{ccc}
     1 & -1 & 1 \\
     1 & -1 & -1 \\
     0 & 2 & 0 \\
   \end{array}
 \right)
,$$ moving the axis $\mathcal{L}:X-Y=0$ of $\phi:X\leftrightarrow Y,
Z\mapsto Z$ to $Z=0$. In particular, we obtain the
$\overline{k}$-equivalent model
$$Z^4-2Z^2\left(X^2-8XY+(2a+7)Y^2\right)+\left(X^4+2(2a-3)(XY)^2-8(a-1)XY^3+(4a+16b-3)Y^4\right).$$
\end{rem}

\noindent\textit{Proof of Theorem \ref{biellipticquartics}}. Given a smooth $(k,\overline{k})$-plane quartic curve $C$ defined by the form
$Z^4+Z^2L_{2,Z}+L_{4,Z}=0$ over $\overline{k}$, the map $C_{\overline{k}}\rightarrow C_{\overline{k}}/\langle\operatorname{diag}(1,1,-1)\rangle$
is, by Riemann-Hurwitz formula, a two-to-one $\overline{k}$-morphism to a genus one curve over $\overline{k}$. Hence,
$C_{\overline{k}}/\langle\operatorname{diag}(1,1,-1)\rangle$ is an elliptic over
$\overline{k}$, and $C$ is bielliptic.

Conversely, let $C$ be a smooth plane quartic curve over $k$,
which is bielliptic. In particular, $C_{\overline{k}}$ admits an
order two automorphism that can be taken, up to
$\overline{k}$-projective equivalence, as
$\tilde{w}=\operatorname{diag}(1,1,-1)$ leaving invariant a
non-singular plane model $F(X,Y,Z)=0$ of $C_{\overline{k}}$. By
non-singularity, $F(X,Y,Z)$ should be of degree at least $3$ in each
variable. Consequently, $F(X,Y,Z)$ reduces to the one of the forms
$Z^4+Z^2L_{2,Z}(X,Y)+L_{4,Z}(X,Y)=0$ or $Z^3L_{1,Z}+ZL_{3,Z}=0$ (cf. \cite{BaBa2016IA}).
However, the latter case is absurd, since it decomposes into $Z\cdot
G(X,Y,Z)$ and it becomes singular.

The stratification by automorphism groups follows from the work of
Henn \cite{He} (see also \cite{Ba2}), and for the families being
geometrically complete, we refer to \cite{Ritothers, PhDLorenzo}. \quad\quad\quad\quad\quad\quad\quad\quad\quad\quad\quad\quad\quad\quad\quad\quad\quad\quad\quad\quad\quad\quad\quad\quad$\square$

Be noted that any of the strata
$\widetilde{\mathcal{M}_3^{\operatorname{Pl}}}(G)$, whenever it is
non-empty, corresponds to a unique representation
$\varrho:G\hookrightarrow\operatorname{PGL}_3(\overline{k})$.
Consequently, any two smooth plane quartic curves with automorphism
groups isomorphic to $G$ are also $\overline{k}$-isomorphic (this is
not the case for higher degrees, since there are smooth
$(k,\overline{k})$-plane curves $C,
C'$ of the same degree $d>4$ with isomorphic but
non-conjugated automorphism groups, in particular,
$C_{\overline{k}}$ and $C'_{\overline{k}}$ are not
$\overline{k}$-isomorphic (cf. \cite{BadrThesis, BaBamed})).
Accordingly, by the aid of $\operatorname{GAP}$ library \cite{gap},
we can list down all bielliptic involutions that may happen by
considering the fixed $\varrho(G)$ given by Henn \cite{He} (cf.
\cite{Ba2, Ritothers, PhDLorenzo}):
\subsection*{Notations} We use $\zeta_n$ for a fixed primitive $n$th root of unity
inside $\overline{k}$ when the characteristic of $k$ is coprime with $n$. A projective linear transformation $A=(a_{i,j})$ of the plane $\mathbb{P}^2_{\overline{k}}$ is often written as $$[a_{1,1}X+a_{1,2}Y+a_{1,3}Z:a_{2,1}X+a_{2,2}Y+a_{2,3}Z:
a_{3,1}X+a_{3,2}Y+a_{3,3}Z],$$
where $\{X,Y,Z\}$ are the homogenous coordinates of $\mathbb{P}^{2}_{\overline{k}}$.

 \begin{cor}[Bielliptic involutions]\label{lem2.6} Let $C$ be {a smooth $(k,\overline{k})$-plane quartic curve over a field $k$} of characteristic $p=0$ or $p>7$. Assume that $C$ is bielliptic, then the set of bielliptic involutions, acting on a non-singular plane model for $C_{\overline{k}}$ in one of the families of Theorem \ref{biellipticquartics}, are classified as follows:
\begin{enumerate}[(i)]
  \item when $G=\Z/2\Z$ or $\Z/6\Z$; $\operatorname{diag}(1,1,-1)$,
\item when $G=\Z/2\Z\times\Z/2\Z$; $\operatorname{diag}(1,1,-1),\,\operatorname{diag}(1,-1,1),$ and $\operatorname{diag}(-1,1,1)$,
  \item when $G=\operatorname{D}_4$; $\operatorname{diag}(1,1,-1),\,[Y:X:\pm Z],$ and $[Y:-X:\pm\zeta_4Z]$,
  \item when $G=\operatorname{S}_3$; $[Y:X:Z],\,[\zeta_3Y:\zeta_3^2X:Z],$ and $[\zeta_3^2Y:\zeta_3X:Z]$,
  \item when $G=\operatorname{GAP}(16,13)$; $\operatorname{diag}(1,1,-1),\,\operatorname{diag}(1,-1,1),\,\operatorname{diag}(-1,1,1),\,[Y:X:\pm Z],$ and $[Y:-X:\pm\zeta_4Z]$,
  \item when $G=\operatorname{S}_4$; $\operatorname{diag}(1,1,-1),\,\operatorname{diag}(1,-1,1),\,\operatorname{diag}(-1,1,1),\,[Y:X:\pm Z],\,[Z:\pm Y:X],$ and $[\pm X:Z:Y]$,
\item when $G=\operatorname{GAP}(48,33)$; $\operatorname{diag}(1,1,-1),\,\operatorname{diag}(1,-1,1),\,\operatorname{diag}(-1,1,1),\,[Y:X:\pm Z],$ and $[Y:-X:\pm\zeta_4Z],$
\item when $G=\operatorname{GAP}(96,64)$; $\operatorname{diag}(1,1,-1),\,\operatorname{diag}(1,-1,1),\,\operatorname{diag}(-1,1,1),\,[Y:X:\pm Z],\,[Y:-X:\pm\zeta_4Z],\,[Z:\pm Y:X],\,[\pm X:Z:Y],\,[Z:\pm\zeta_4Y:-X],$ and $[\pm\zeta_4X:-Z:Y]$,
\item when $G=\operatorname{PSL}_2(\mathbb{F}_7)$; Let $\alpha:=\frac{-1+\sqrt{-7}}{2}$, and
$$g:=\left(
       \begin{array}{ccc}
         -2 & \alpha & -1 \\
         \alpha & -1 & 1-\alpha \\
         -1 & 1-\alpha & -1-\alpha \\
       \end{array}
     \right)
,\,h:=\left(
        \begin{array}{ccc}
          0 & 1 & 0 \\
          0 & 0 & 1 \\
          1 & 0 & 0 \\
        \end{array}
      \right)
,\,s:=\left(
        \begin{array}{ccc}
          -3 & -6 & 2 \\
          -6 & 2 & -3 \\
          2 & -3 & -6 \\
        \end{array}
      \right).$$
Then, we get the $21$ involutions $\psi\circ(\phi_0\circ s\circ\phi_0^{-1})$ for $\psi\in\langle g,h\rangle$, where
$$\phi_0:=\left(
           \begin{array}{ccc}
             1 & 1+\zeta_7\alpha & \zeta_7^2+\zeta_7^6 \\
             1+\zeta_7\alpha & \zeta_7^2+\zeta_7^6 & 1 \\
             \zeta_7^2+\zeta_7^6 & 1 & 1+\zeta_7\alpha \\
           \end{array}
         \right)\circ\left(
                       \begin{array}{ccc}
                         -\alpha & 1 & 2\alpha+3 \\
                         2\alpha+3 & -\alpha & 1 \\
                         1 & 2\alpha+3 & -\alpha \\
                       \end{array}
                     \right).$$
                     \end{enumerate}
\end{cor}

\end{section}

\begin{rem} One observes that for smooth plane quartic curves that are bielliptic, all bielliptic
involutions are defined over a finite separable extension of the base field $k$, up to conjugation by a linear projective transformation of the plane.
\end{rem}
\subsection{Infinitude of quadratic points}

\begin{lem}\label{Tate} Let $C_k$ be a smooth plane curve of degree $d\geq4$ over $k$, where $k$ is a global field of characteristic $p>(d-1)(d-2)+1$. Then, $C_k$ is conservative over $k$.
\end{lem}

\begin{proof}
The result follows from \cite[Corollary 2]{Tate}, since the relative
genus $g_{C,k}=(d-1)(d-2)/2<(p-1)/2$.
\end{proof}

\begin{thm}\label{thm3.3} Let $C$ be a smooth plane curve of degree $d\geq4$ over a global field $k$ of characteristic $p=0$ or $p>(d-1)(d-2)+1$, and assume in positive
characteristic that $Jac(C_k)$ over $\overline{k}$ has no non-zero
homomorphic images defined over $\overline{\F}_q$. For any finite
field extension $L/k$ inside $\overline{k}$, the set of quadratic
points $\Gamma_2(C,L)$ of $C$ over $L$ is always a finite set when
$d\geq5$, also it does when $d=4$ and $\Aut(C_{\overline{k}})\cong
1, \Z/3\Z,$ or $\Z/9\Z$. Moreover,
\begin{enumerate}[(i)]
\item if $d=4$ and $p=0$, then it exists a number field $L/k$ inside $\overline{k}$ for which $\Gamma_2(C,L)$ is an infinite set
if and only if $\Aut(C_{\overline{k}})\cong\Z/2\Z,\,\Z/2\Z\times\Z/2\Z,\,\Z/6\Z,\,\operatorname{S}_3,\,\operatorname{D}_4,\,\operatorname{GAP}(16,13),\,\operatorname{S}_4,\, \operatorname{GAP}(48,33),\, \operatorname{GAP}(96,64),$ or $\operatorname{PSL}_2(\F_7)$.
\item if $d=4$ and $p>7$, then $\operatorname{Jac}(C)$ over $\overline{k}$ does not contain bielliptic quotients associated to isotrivial elliptic curves {only if} $C_{\overline{k}}$ belongs to an open set (of the same dimension) of one of the strata
$\widetilde{\mathcal{M}_3^{\operatorname{Pl}}}(G)$ with
$G\in\{\Z/2\Z,\,\Z/2\Z\times\Z/2\Z,\,\operatorname{S}_3,\,
\operatorname{D}_4,\,\operatorname{GAP}(16,13),\,\operatorname{S}_4\}$
\footnote{{By the virtue of the theorem of Grauert-Samuel in \S1,
the other non-empty strata $\widetilde{M_3^{\operatorname{Pl}}}(G)$
not satisfying the hypothesis of Theorem \ref{thm3.3} (ii) may also
have infinite number of point without need to extend to a degree $2$
extension.}}. Moreover, for such $C$'s (i.e when no isotrivial
elliptic curves appear), there is a finite field extension $L/k$
inside $\overline{k}$ such that $\Gamma_2(C,L)$ is an infinite set.
\end{enumerate}
\end{thm}
\begin{proof} First, when $d\geq5$ or $d=4$ and $\Aut(C_{\overline{k}})\cong 1, \Z/3\Z,$ or $\Z/9\Z$, $C$ is neither hyperelliptic nor
bielliptic (Proposition \ref{nonhypergenus} and Theorem
\ref{biellipticquartics}). Therefore, $\Gamma_2(C,L)$ is a finite
set (Theorems \ref{SchweizerII}, \ref{AbramovichHarrisSilvermanII}).
Second, if $d=4$ and $p=0$, then, by Theorem
\ref{biellipticquartics}, $C$ is bielliptic if and only if
$\Aut(C_{\overline{k}})\cong\Z/2\Z,\,\Z/2\Z\times\Z/2\Z,\,\Z/6\Z,\,\operatorname{S}_3,\,\operatorname{D}_4,\,\operatorname{GAP}(16,13),\,\operatorname{S}_4,\,
\operatorname{GAP}(48,33),\, \operatorname{GAP}(96,64),$ or
$\operatorname{PSL}_2(\F_7)$. Equivalently, $\Gamma_2(C,L)$ is an
infinite set for some finite extension $k\subseteq
L\subseteq\overline{k}$ (Theorem \ref{AbramovichHarrisSilvermanII}).
Finally, assume that $d=4$ and $p>7$, then, by Lemma \ref{Tate}, $C$
is conservative over $k$. If
$\operatorname{Aut}(C_{\overline{k}})=\Z/6\Z,\,\operatorname{GAP}(48,33),\,\operatorname{GAP}(96,64),$
or $\operatorname{PSL}_2(\F_7)$, then isotrivial elliptic curves
appear in $\operatorname{Jac}(C)$ over $\overline{k}$ (for more
details, see Lemma \ref{lem4.8} and its proof, in particular the
MAGMA code to compute the Jacobian and $j$-invariants), hence we can
not apply Theorem \ref{SchweizerII} in these situation.
 Furthermore, in some cases, the $j$-invariant depends on parameters where some particular specializations corresponds to isotrivial elliptic curves, but not the generic case (we again refer to Lemma \ref{lem4.8}), and this justifies $C_{\overline{k}}$ being a member of an open set of one of the prescribed strata $\widetilde{\mathcal{M}_3^{\operatorname{Pl}}}(G)$. This shows the ``only if'' part. For the remaining situations of $\widetilde{\mathcal{M}_3^{\operatorname{Pl}}}(G)$ where
 isotrivial elliptic curves do not occur, we may apply Proposition \ref{100} to deduce that $\Gamma_2(C,L)$ is infinite for some finite field extension $L/k$ inside $\overline{k}$.
 \end{proof}

The next result maybe is well-known to specialists, however we
present a quite easy proof.
\begin{cor} \label{FermatKlein} Let $d\geq5$ be a fixed integer, and $C$ be any smooth plane curve of degree $d$ over a global field $k$ of characteristic $p=0$ or $p>(d-1)(d-2)+1$
(and in positive characteristic we further assume that $\operatorname{Jac}(C_k)$ over
$\overline{k}$ has no non-zero homomorphic images defined over
$\mathbb{F}_q$). Then, the number of quadratic field extensions
$k\subset k'\subseteq\overline{k}$ where $C(k')\neq C(k)$ are
finitely many. In particular, if we consider the Fermat curve
$C:X^d+Y^d-Z^d=0$ of degree $d$ over $\Q$, then
$C(\Q)=C(\Q(\sqrt{D}))$ for all square-free integers $D$, except
possibly finitely many values for $D$. That is, there are, in the
worst case, only finitely many quadratic number fields in which we
may have more solutions in $\Q(\sqrt{D})$ to $X^d+Y^d-Z^d=0$ than
these over $\Q$. The same is true for the Klein curve
$X^{d-1}Y+Y^{d-1}Z+Z^{d-1}X=0$, and any other smooth plane curves of
degree $d$ over $\Q$ or number fields.
\end{cor}

\begin{proof} By definition, we have $C(k)\subseteq C(k')\subseteq\Gamma(C,k)$ for any quadratic field extension $k\subset k'\subseteq\overline{k}$. We also know from Theorem \ref{thm3.3} that $\Gamma_2(C,k)$ is a finite set. Hence, only finitely many $k'$ could satisfy $C(k)\subsetneq C(k')$.
\end{proof}


\begin{section}{Quadratic points on smooth plane curves fixing the base field: case $\Q$}
In this section, we restrict our attention to smooth $(\Q,\overline{\Q})$-plane quartic curves $C$, which are bielliptic (see Theorem \ref{biellipticquartics}). Since, the degree $d=4$ is relatively prime to $3$, we have that $C$ is a smooth plane curve over $\Q$, that is, $\Q$ is not only a field of definition for $C$, but also a plane-model field of definition.

We find some interest to conjecture the following.
\begin{conj}\label{conj} Fix a stratum of the shape $\widetilde{\mathcal{M}_3^{\operatorname{Pl}}}(G)$, where all of its
$\overline{\mathbb{Q}}$-points are bielliptic, equivalently, take $G\neq1,\Z/3\Z,\Z/9\Z$. Then, there is an infinite family $\mathcal{E}$ (resp. $\mathcal{D}$)$\subseteq\widetilde{\mathcal{M}_3^{\operatorname{Pl}}}(G)$ of $\Q$-isomorphism classes of smooth plane quartic curves over $\mathbb{Q}$, such that $\Gamma_2(C,\mathbb{Q})$ is a finite (resp. infinite) set, for all $[C]\in\mathcal{E}$.
\end{conj}
In what follows, we are going to support the above conjecture in two
different situations. By the work of
Lercier-Ritzenthaler-Rovetta-Sijsling in \cite[\S2]{Ritothers}, we
have a parametrization of the (coarse) moduli space of smooth plane
quartic curves in terms of complete and representative families over
$\Q$ for all the strata
$\widetilde{\mathcal{M}_3^{\operatorname{Pl}}}(G)$, except when
$G=\Z/2\Z$ (a representative family over $\mathbb{R}$ does not
exist). Moreover, the work of Lorenzo in \cite{PhDLorenzo} detailed
the study of the twists of smooth plane quartic curves over $\Q$.
These results helps us a lot in our study of quadratic points in the
sense of the previous conjecture. The main idea is to start with a
family of smooth plane quartic curves over $\Q$, having many
parameters in the defining equation over $\Q$. This in turns allows
us to construct a subfamily with infinite cardinality of non
$\Q$-isomorphic smooth plane quartic curves with the same
automorphism group $G$ (up to group isomorphisms), where its members
are mapped to concrete elliptic curves over $\Q$ whose rank is zero
(resp. positive) over $\Q$. We
note that, in the first case, we may reduce up to change of
variables to some twists over $\Q$ that have a
unique bielliptic involution defined over $\Q$ (in particular, it
suffices to deal with a specific elliptic curve of rank zero, not
with family of a elliptic curves over $\Q$).

\begin{subsection}{The conjecture \ref{conj} is true for $\widetilde{\mathcal{M}_3^{\operatorname{Pl}}}(\Z/6\Z)$}
Let $k$ be a field of characteristic $p=0$ or $p>7$. The one parameter family $\mathcal{C}_a$ defined by
$$\mathcal{C}_a:\,aZ^4+Y^2(Y^2+aZ^2)+X^3Y=0,$$
where $a\neq0,4$, is a representative family over $k$ for $\widetilde{\mathcal{M}_3^{\operatorname{Pl}}}(\Z/6\Z)$, see \cite[Theorem 3.3]{Ritothers}. Thus, any smooth plane quartic curve $C$ over $k$ with automorphism group isomorphic to $\Z/6\Z$ has a non-singular plane model in $\mathcal{C}_a$ for a unique $a\in k$, that is, there exists a finite extension $L/k$ inside $\overline{k}$ where $C\otimes_kL$ is $L$-isomorphic to a unique non-singular polynomial equation of the form $aZ^4+Y^2(Y^2+aZ^2)+X^3Y=0$ for some $a\in k$. In particular, $\Gamma_2(C,L)=\Gamma_2(aZ^4+Y^2(Y^2+aZ^2)+X^3Y=0,L)$.

We also note that the above family is
$\overline{k}$-isomorphic to the geometrically complete family in
Theorem \ref{biellipticquartics} via a diagonal change of variables
\cite{Ritothers}, hence by Corollary \ref{lem2.6},
$\tilde{w}=\operatorname{diag}(1,1,-1)$ is again a unique
bielliptic involution for any smooth curve in the family
$\mathcal{C}_a$.

\begin{thm}\label{thm1} Consider a smooth bielliptic quartic plane curve $C_a$ of the form $aZ^4+Y^2(Y^2+aZ^2)+X^3Y=0$ for some $a\in\mathbb{Q}\setminus\{0,4\}$. Then, the quotient
$C_a/\langle\tilde{w}\rangle$ is an elliptic curve of
positive rank over $\Q$. In particular, $\mathcal{C}_a$ with $a\in
\mathbb{Q}\setminus\{0,4\}$ is an infinite family of bielliptic smooth plane
quartic curves over $\Q$ whose full automorphism group isomorphic to
$\Z/6\Z$, and such that $\Gamma_2(\mathcal{C}_a,\Q)$ is an infinite
set.
\end{thm}
\begin{proof} We work affine with $az^4+az^2+x^3+1=0$ by taking $Y=1$ (observe that, we have a unique point on $C_a$ with
$Y=0$ that is also fixed by $\tilde{w}$). In particular, $C_a/\langle\tilde{w}\rangle:\,az^2+az+x^3+1=0$. We apply the
$\Q$-change of variables $x\mapsto-x$ and $z\mapsto z-\frac{1}{2}$
to obtain $az^2=x^3+a/4-1$. Next, change $z\mapsto(1/a^2)z$ and
$x\mapsto(1/a)x$ to finally get the elliptic curve
$E/\Q:z^2=x^3-a^3(1-a/4)$ whose $j$-invariant equals to zero.
Furthermore, the point $P_a:=(x,z)=(a,a^2/2)$ is a non-torsion point
on $E/\Q$ (if it is not, then it has order $m\leq10$ or $m=12$ (cf.
\cite[Theroem 7']{Mazur}), and one can check by MAGMA that the order
of $P_a$ is distinct from these values). Consequently,
$\operatorname{rank}_{\Q}(C_a/\langle\tilde{w}\rangle)\geq 1$ for
any $a\in\Q\setminus\{0,4\}$, and thus $\Gamma_2(C_a,\Q)$ is an
infinite set (Theorem \ref{AbramovichHarrisSilvermanI}). Finally,
the family $\mathcal{C}_a$ is representative for
$\widetilde{\mathcal{M}_3^{\operatorname{Pl}}}(\Z/6\Z)$ over $\Q$,
that is $C_a$ and $C_{a'}$ in the family with
$a\neq a'\in\mathbb{Q}$ can not be
$\overline{\mathbb{Q}}$-isomorphic. Therefore, we get infinitely
many smooth plane quartic curves that have infinitely many quadratic
points over $\Q$.
\end{proof}

Now, if we are interested to parameterize the set of $k$-isomorphism classes of smooth plane quartic curves over $k$ with automorphism group isomorphic to $\Z/6\Z$, then we may use \cite[Proposition 3.2.8]{PhDLorenzo}
 to have that the family with $A,n,m\in k^*$;
$$\mathcal{C}_{A,n,m}:Am^2Z^4+mY^2Z^2+nX^3Y+Y^4=0,$$
where $(A,n,m)\in k^*\times k^*/ k^{*^3}\times k^*/ k^{*^2}$
does\footnote{It remains to determine the algebraic restrictions on
the parameters to ensure non-singularity and no larger automorphism
group. For example, $A\neq1/4$, since we get the singular points
$(0:\pm\sqrt{m/2}:1)$ on $\mathcal{C}_{1/4,n,m}$, otherwise.}. That is, any smooth plane quartic curve over $k$ with
automorphism group $\Z/6\Z$ is $k$-isomorphic to a non-singular
plane model in the family $\mathcal{C}_{A,n,m}$ for some triple
$(A,n,m)\in k^*\times k^*/ k^{*^3}\times k^*/ k^{*^2}$, in
particular, it is bielliptic with unique bielliptic involution
$\tilde{w}=\operatorname{diag}(1,1,-1)$, by Corollary \ref{lem2.6}.
Moreover, two such curves are $\overline{k}$-isomorphic if and only
if they have the same parameter $A\in k^*$. Therefore, it suffices
to assume $\Gamma_2(\mathcal{C}_{A,n,m},k)$ for $(A,n,m)\in
k^*\times k^*/ k^{*^3}\times k^*/ k^{*^2}$ in order to investigate
the infinitude of quadratic points over $k$ of smooth plane quartic
curves inside
$\widetilde{\mathcal{M}_3^{\operatorname{Pl}}}(\Z/6\Z)$ over $k$.

\begin{lema} For an arbitrary but a fixed $(A,n)\in k^*\times k^*/ k^{*^3}$, the set $\Gamma_2(\mathcal{C}_{A,n,m},k)$ being finite or infinite is independent of the choice of $m$.
\end{lema}

\begin{proof}
One finds that $\mathcal{C}_{A,n,m}/\langle\tilde{w}\rangle$ is
$k$-isomorphic to the elliptic curve
$\mathcal{D}_{A,n}/k:\,z^2=x^3+\frac{n^2A^2(1-4A)}{4}$ over $k$. Indeed, one works affine $Y=1$ to easily obtain that $\mathcal{C}_{A,n,m}/\langle\tilde{w}\rangle$ is $k$-isomorphic to the elliptic curve $Az^2+z+nx^3+1=0$ over $k$, in particular, its defining equation and its rank is independent from $m$. To reach the Weierstrass form
$\mathcal{D}_{A,n}$, we follow the usual way (cf. \cite{Silv1}); first, by the change of variables
$z\mapsto z-\frac{1}{2A}$ and $x\mapsto-x$, we get $z^2=(n/A)x^3+\frac{1-4A}{4A^2}$, and after by $z\mapsto (1/nA^2)z$ and $x\mapsto (1/An)x$.
\end{proof}

\begin{thm}\label{thm2} There are infinitely many smooth plane quartic curves $C$ over $\Q$ in the family $\mathcal{C}_{A,n,m}$, with $(A,n,m)\in\Q^*\times\Q^*/{\Q^*}^3\times\Q^*/{\Q^*}^2$,
 such that $\Gamma_2(C,\Q)$ is a finite set.
\end{thm}
\begin{proof}
For example, we may consider the subfamily
$\mathcal{C}_{A(t),n(t),m}:A(t)m^2Z^4+mY^2Z^2+n(t)X^3Y+Y^4=0$ where
$A(t):=(108t^2+1)/4$ and $n(t):=4/t(108t^2+1)$, for $t=a/b\in\Q^*$
with $a,b$ odd relatively prime integers (note that
$n(t)\in\Q^{*^3}$ only if either $a$ or $b$ is even).
In this situation,
$\mathcal{C}_{A(t),n(t),m}/\langle\tilde{w}\rangle$ is always
$\Q$-isomorphic to $\mathcal{D}_{A(t),n(t)}:z^2=x^3-27$, which has rank $0$ over $\Q$.
Since, two curves in the family associated to the triples
$(A(t),n(t),m)$ and $(A(t'),n(t'),m')$ with $A(t)\neq A(t')$ are not
$\overline{\mathbb{Q}}$-isomorphic, we get infinitely many smooth
plane quartic curves over $\Q$ that have infinitely many quadratic
points over $\Q$, which was to be shown.
\end{proof}

\end{subsection}

\begin{subsection}{The conjecture \ref{conj} is true for
$\widetilde{\mathcal{M}_3^{\operatorname{Pl}}}(\operatorname{GAP}(16,13))$} The family defined by
$$\mathcal{C}_a:Z^4+(X^4+Y^4+aX^2Y^2)=0,$$
where $\pm a\neq 0,2,6,2\sqrt{-3}$, is a geometrically complete family for the stratum  $\widetilde{\mathcal{M}_3^{\operatorname{Pl}}}(\operatorname{GAP}(16,13))$ over $\overline{\Q}$.
As already mentioned in Corollary \ref{lem2.6}, we have exactly seven bielliptic involutions, namely
$\iota_1:=\operatorname{diag}(1,1,-1),\,\iota_2:=\operatorname{diag}(-1,1,1),\,\iota_3:=\operatorname{diag}(1,-1,1),\,\iota_4:=[Y:X:Z],\,\iota_5:=[Y:X:-Z],\, \iota_6:=[Y:-X:\zeta_4 Z]$ and $\iota_7:=[Y:-X:-\zeta_4 Z]$, where $\zeta_4$ is a fixed primitive $4$th root of unity in $\overline{\Q}$.

\begin{lem}\label{lem4.8} For $a\in\Q\setminus\{0,\pm2,\pm6\}$ and assuming that $\mathcal{C}_a/\langle\iota_i\rangle$ has a $\Q$-point, then it is  $\Q$-isomorphic to an elliptic curve $E/\Q$ of one of the forms:
\begin{center}
\begin{tabular}{|c|c|c|}

\hline
Involution& $E$ & $j$-invariant\\
\hline  $\iota_1$&$z^2=x^3-ax^2-4x+4a$&$(16a^6+576a^4+6912a^2+27648)/(a^4-8a^2+16)$
\\
\hline
$\iota_2,\iota_3$&$z^2=x^3+(a^2-4)x$&$1728$\\
\hline $\iota_4,\iota_5,\iota_6,\iota_7$ & $z^2=x^3+(1-a^2/4)x$ & $1728$\\\hline
\end{tabular}
\end{center}

\end{lem}
\begin{proof}
The next MAGMA code applied for
$\mathcal{C}_a/\langle\iota_6\rangle$ can also be adapted elsewhere.
It assures that $\mathcal{C}_a/\langle\iota_6\rangle$ is a smooth
curve of genus $1$ over $\Q$. Because
$\mathcal{C}_a/\langle\iota_i\rangle$ has rational points, it is
$\Q$-isomorphic to its Jacobian variety over $\Q$, which is an
elliptic curve over $\Q$, a priori.

\texttt{> R<x>:=PolynomialRing(Integers());}

\texttt{> K<k>:=NumberField(x$^\texttt{2}$+$\texttt{1}$);}

\texttt{> L<a>:=FunctionField(K,1);}

\texttt{> P2<X,Y,Z>:=ProjectiveSpace(L,2);}

\texttt{> g:=X$^\texttt{4}$+Y$^\texttt{4}$+Z$^\texttt{4}$+aX$^\texttt{2}$Y$^\texttt{2}$;}

\texttt{> C:=Curve(P2,g);}

\texttt{> phi1:=iso<C->C|[Y,-X,k*Z],[Y,-X,k*Z]>;}

\texttt{> G1:=AutomorphismGroup(C,[phi1]);}

\texttt{> CG1,prj:=CurveQuotient(G1);}

\texttt{> Genus(CG1);}

\texttt{1}

\texttt{> Jacobian(CG1); E:=Jacobian(CG1);}

\texttt{Elliptic Curve defined by y$^\texttt{2}$ = x$^\texttt{3}$+($\texttt{-1/4}$a$^\texttt{2}$+$\texttt{1}$)x over Multivariate
rational function field}

\texttt{of rank 1 over K}

\texttt{> jInvariant(E);}

\texttt{$1728$}

The condition that $\mathcal{C}_a/\langle\iota_i\rangle$ has
rational points is verified in many cases. For example, we can see
that $\mathcal{C}_a/\langle\iota_1\rangle$ is defined inside
$\mathbb{P}^3_{\overline{\Q}}$ by the two quadrics
$-X_2X_3+X_4^2=0,\,\,\text{and}\,\,X_1^2+X_2^2+X_3^2+aX_2X_3=0$ over
$\Q$. Hence, if we impose $-(a+2)\in\mathbb{Q}^{*^2}\setminus\{4\}$,
then $(\sqrt{-(a+2)}:1:1:1)$ is an obvious $\Q$-point.
\end{proof}

Consider the family $\mathcal{C}'_A: AX^4+Y^4+Z^4+X^2Y^2=0$, where
$\pm A\neq\{1/4,1/36,1/12\}$. This is a representative family over
$\Q$ (cf. \cite[p. 36,37]{PhDLorenzo}), in particular, any smooth
plane quartic curve over $\Q$ with automorphism group isomorphic to
$\operatorname{GAP}(16,13)$ is isomorphic (not necessarily over
$\Q$) to a smooth curve $C_A:AX^4+Y^4+Z^4+X^2Y^2=0$ in the family
$\mathcal{C}_A'$ for an unique
$A\in\Q^*\setminus\{\pm1/4,\pm1/36,\pm1/12\}$.

The transformed seven bielliptic involutions of any smooth plane
quartic curve in the family $\mathcal{C}'_A$ over $\Q$ are
$\iota_i':=P^{-1}\iota_i P$, for $i=1,\ldots,7$, where
$$P:=\left(\begin{array}{ccc}
\frac{1}{\sqrt[4]{A}}&0&0\\
0&1&0\\
0&0&1\\
\end{array}\right)$$
We also impose, once and for all, that $A\in \Q^*\setminus\Q^{*^4}$. In particular,
$\iota_1'=\iota_1,\,\iota_2'=\iota_2,\,\iota_3'=\iota_3$ are the
only bielliptic involutions defined over $\Q$.

By the work of E. Lorenzo Garc\'{\i}a in \cite[Chp. 3]{PhDLorenzo},
we know that any \emph{diagonal twist} of $C_A$, for a fixed $A$, in the family
$\mathcal{C}_A'$ is $\Q$-isomorphic to
$$C_{A,m,q}: mAX^4+q^2mY^4+Z^4+qmX^2Y^2=0,$$ {for some $A,m,q\in\Q^*$, where two of twists $\{A,m,q\}$ and $\{A',m',q'\}$ are $\Q$-isomorphic if $A=A', m\equiv\,m'\, \operatorname{mod}\,\Q^{*^4}$ and $q \equiv\,q'\,
\operatorname{mod}\,\Q^{*^2}$. First, we consider smooth curves of the form $C_{A,m,q}$ with $(A,m,q)\in
\Q^*\times\Q^*/\Q^{*^4}\times\Q^*/\Q^{*^2}$}. Next, the quotient
curve $C_{A,m,q}/\langle\iota_3'\rangle$ is a genus one curve
$\Q$-isomorphic to $y^2+(1/m)z^4-(1/4-A)=0$, in particular it is
independent of the parameter $q\in \Q^*/ \Q^{*^2}$. Let us assume
that $f:=1/4-A$ is a square (hence,
$C_{A,m,q}/\langle\iota_3'\rangle$ has a $\Q$-point), and after one
may use Maple's \textbf{Weierstrassform} function

\texttt{> algcurves:-Weierstrassform(y$^\texttt{2}$ - (-1/m)z$^\texttt{4}$ - f, x, y, u, v);}

which returns the normal form in variables $u,\nu$:
$$\nu^2=u^3+((1-4A)/m)u.$$

\begin{thm}\label{thm3}
Consider a smooth bielliptic quartic plane curve in the family $\mathcal{C}_{A,m,q}: mAX^4+q^2mY^4+Z^4+qmX^2Y^2=0,$
for some $q\in \Q^*\setminus\Q^{*^2}$ and $A, m\in \Q^*\setminus\Q^{*^4}$ such that $1/4-A$ is a square. Then, the quotient $\mathcal{C}_{A,m,q}/\langle\tilde{w}\rangle$, where $\tilde{w}=\operatorname{diag}(1,-1,1)$, is an elliptic curve of positive rank over $\Q$. In particular, $\mathcal{C}_{A,m,q}$ gives an infinite family of smooth plane quartic curves $C$ over $\Q$ whose automorphism group is isomorphic to $\operatorname{GAP}(16,13)$, and $\Gamma_2(C,\Q)$ is an infinite set.
\end{thm}

\begin{proof}
We have seen above that
$\mathcal{C}_{A,m,q}/\langle\tilde{w}\rangle$ is $\Q$-isomorphic to
$\mathcal{D}_{A,m}:\nu^2=u^3+Du$ with $D:=(1-4A)/m$ independently
from the parameter $q\in\Q^*/\Q^{*^2}$. It now suffices to
specialize $A,m$ accordingly so that
$\operatorname{rank}_{\Q}(\mathcal{D}_{A,m})$ is positive, hence
$\Gamma_2(C,\Q)$ is infinite. For example, take $D:=(1-4A)/m=p$,
where $p$ is a prime integer $<1000$ and congruent to $5$ modulo
$8$, so the rank is always $1$ in accordance with the conjecture of
Selmer and Mordell (see \cite{Bremner-Cassels}). Take the case
$D:=(1-4A)/m=-p$, where $p$ is a Fermat or a Mersenne prime, then
the ranks $0,\,1$ and $2$ were found (see \cite{Kudo-Motose}). Take
$D:=(1-4A)/m=-n$, where $n$ is related to the positive integer
solutions of the diophantine equation
$n=\alpha^4+\beta^4=\gamma^4+\mu^4$, then the rank is at least $3$
(see \cite{Izadi-Khoshnam-Nabardi}). Take $D:=(1-4A)/m=-pq$, where
$p$ and $q$ are two different odd primes, then, up to an extra
condition, a family of rank $4$ was found (see \cite{Maenishi}),
etc... Finally, by \cite[Proposition
3.2.9]{PhDLorenzo}, two twists $C_{A,q,m}$ and $C_{A,q',m'}$, with
$(A,m,q),(A,m',q')\in\Q^*\times\Q^*/\Q^{*^2}\times\Q^*/\Q^{*^4}$
are $\Q$-isomorphic only if $m=m'$ in $\Q^*/\Q^{*^4}$. Consequently,
for any fixed $A$, we can run $m$ and $q$ as before to obtain
infinitely many non-$\Q$-isomorphic smooth plane quartic curves with infinitely many quadratic points over $\Q$.
\end{proof}

Now, we ask for an infinite family of smooth quartic curves over
$\Q$ in the stratum
$\widetilde{\mathcal{M}_3^{\operatorname{Pl}}}(\operatorname{GAP}(16,13))$
such that the quotient by any bielliptic involutions may only
provide elliptic curves of rank $0$ over $\Q$. In particular, the
set quadratic points over $\Q$ is finite. To do so, we turn out to
\emph{non-diagonal} twists of $C_A$, for a fixed $A\in\Q^*$, that
are parameterized (see \cite[Proposition 3.2.9]{PhDLorenzo}) by
$$\mathcal{C}_{\underline{a},\underline{b},m,q,A}:
2\underline{a}X^4+8\underline{b}mX^3Y+12\underline{a}mX^2Y^2+8\underline{b}m^2XY^3+2\underline{a}m^2Y^4+q(X^2-mY^2)^2+Z^4=0$$
where $m\in \Q^*$, $\underline{a},\underline{b},q\in \Q$ satisfy
$\underline{a}^2-\underline{b}^2m=q^4A$. Two such twists
$\{\underline{a},\underline{b},m\}$ and
$\{\underline{a}',\underline{b}',m'\}$ for $C_A$ are equivalent if
and only if $m\equiv m' \operatorname{mod} {\Q^*}^2$ and that there
exist $c,d\in \Q$ such that
$\underline{a}+\underline{b}\sqrt{m}=(c+d\sqrt{m})^4(\underline{a}'+\underline{b}'\sqrt{m})$.

The transformed seven bielliptic involutions of any smooth plane
quartic curve $C_{\underline{a},\underline{b},m,q,A}$ in the family
$\mathcal{C}_{\underline{a},\underline{b},m,q,A}$ over $\Q$ are
$\iota_i'':=P'^{-1}\iota_i' P'$, for $i=1,\ldots,7$, where
$$P':=\left(\begin{array}{ccc}
\sqrt[4]{\underline{a}+\underline{b}\sqrt{m}}&\sqrt{m}\sqrt[4]{\underline{a}+\underline{b}\sqrt{m}}&0\\
\sqrt[4]{\underline{a}-\underline{b}\sqrt{m}}&-\sqrt{m}\sqrt[4]{\underline{a}-\underline{b}\sqrt{m}}&0\\
0&0&1\\
\end{array}\right)$$
{One can check that} $\operatorname{diag}(1,1,-1)$ is the unique
bielliptic involution defined over $\Q$ for
$C_{\underline{a},\underline{b},m,q,A}$, {when} $A\notin\Q^{*^4}$
and $m\notin\Q^{*^2}$. Thus, the only way to obtain a degree two
$\Q$-morphisms to elliptic curves over $\Q$ is to quotient by
$\tilde{w}:=\operatorname{diag}(1,1,-1)$. Assume moreover that the
quotient family
$\mathcal{C}_{\underline{a},\underline{b},m,q,A}/\langle\tilde{w}\rangle$
has $\Q$-points (for example, we fix $q=-2\underline{a}$ and we
always get the point $(0:1:0)$), hence
$\mathcal{C}_{\underline{a},\underline{b},m,q=-2\underline{a},A}/\langle\tilde{w}\rangle$
is $\Q$-isomorphic to its Jacobian variety, which (by MAGMA) is
given by (for simplicity, we impose $\underline{a}=-\underline{b}$,
which is enough for our proposes):
$$\mathcal{E}_{m,q}:\,y^2=x^3+8qmx^2+16q^2m^3x.$$
With respect to the specialization
$q=-2\underline{a}=2\underline{b}$, we should have
$A=\frac{(1-m)}{4q^2}$ {and we always impose $A\notin\Q^{*^4}$}.

\begin{lem}\label{modular} The two parameter family $\mathcal{E}_{m,q}:\,y^2=x^3+8qmx^2+16q^2m^3x$, for $q\in\mathbb{Q}$ and $m\in \Q^*\setminus{\Q^*}^2$ such that $\frac{1-m}{4q^2}\notin\Q^{*^4}$,
contains infinitely many elliptic curves $E_{m,q}$ of rank $0$ over
$\Q$. More precisely, for any fixed
$m\in\mathbb{Z}\setminus\mathbb{Z}^2$ with
$m\notin\{1-h^2\,|\,h\in\Z\}$, there exist infinitely many $q\in\Q^* \operatorname{mod} \Q^{*^2}$ with $qm$ square-free
integer, and for which $\operatorname{rank}_{\Q}(E_{m,q})=0$.
\end{lem}
\begin{proof} If we specialize $m, mq\in\mathbb{Z}$ so that $m\notin\mathbb{Z}^2$ and $mq$ is square-free, then the elliptic curve $E_{m,q}$ is {$\Q$-isomorphic to the quadratic twist $E_D:Dy^2=x^3+8x^2+16mx$ for $E: y^2=x^3+8x^2+16mx$ over $\Q$, associated to $D=mq$; indeed, if we multiply the defining equation for $E_D$ by $D^3$ and then apply the change of variables $y\mapsto (1/D^2)y$ and
 $x\mapsto (1/D)x$ we obtain the equation of $E_{m,q}$}. We know that $E/\Q$ is modular by the work of so many people around; mainly C. Breuil, B. Conrad, F. Diamond, R. Taylor
 and A. Wiles (cf. \cite{BCDT, CDT, TW, Wiles}). Therefore, we deduce from by L. Main-R. Murty \cite{MaMu} (cf. see also the first line of the abstract \cite{Ono}) that for a
 fixed $m$ there are infinitely many square-free integers $D=mq$ so that $E_D$ has rank $0$ over $\Q$ (each $D$ is congruent to $1$ mod $4\mathfrak{n}$, where $\mathfrak{n}$ is the conductor of the
elliptic curve $E/\Q$). In particular, we get infinitely many
elliptic curves $E_{m,q}$ {(corresponding to infinitely many such
$q$'s $\operatorname{mod}\,\Q^{*^2}$)} of rank $0$ over $\Q$. The
condition that $m\notin\mathbb{Z}^2$ with
$m\notin\{1-h^2\,|\,h\in\Z\}$ is to ensure that $A$ is not a fourth
power.
\end{proof}

\begin{thm}\label{thm4} The two parameter family
$$\mathcal{C}_{m,q}:
-qX^4+4qmX^3Y-6qmX^2Y^2+4qm^2XY^3-qm^2Y^4+q(X^2-mY^2)^2+Z^4=0,$$ for
$q\in \Q^*\setminus\Q^{*^2}$ and $m\in\Q^*/\Q^{*^2}$ such that
$A=\frac{1-m}{4q^2}\notin\Q^{*^4}$, contains
infinitely many, non $\Q$-isomorphic, smooth plane
curves $C$ over $\Q$ that have only finitely many quadratic points
over $\Q$.
\end{thm}
\begin{proof}
If we specialize $q=-2\underline{a}=2\underline{b}$ in the family
$\mathcal{C}_{\underline{a},\underline{b},m,q,A}$ mentioned above,
then we get the subfamily $\mathcal{C}_{m,q}$. In particular, any
smooth plane curve in $\mathcal{C}_{m,q}$ is bielliptic and has only
one bielliptic involution $\tilde{w}$ over $\Q$. Furthermore, by
Lemma \ref{modular}, for any $m\in\mathbb{Z}\setminus\mathbb{Z}^2$
with $m\notin\{1-h^2\,|\,h\in\Z\}$, there exists an infinite number
of $q\in \Q^* \operatorname{mod} \Q^{*^2}$ with $mq$ a
square-free integer, such that
$\mathcal{C}_{m,q}/\langle\tilde{w}\rangle$ is an elliptic curve
over $\Q$ and whose rank is zero. Therefore, the number of quadratic
points over $\Q$ is finite by Theorem
\ref{AbramovichHarrisSilvermanI}. Moreover, if any two curves
$C_{{m,q}}$ and $C_{m,q'}$, which gives rank $0$, are
$\Q$-isomorphic, then
$C_{(1-m)/4q^2}:\frac{1-m}{4q^2}X^4+Y^4+Z^4+X^2Y^2=0$ and
$C_{(1-m)/4q'^2}:\frac{1-m}{4q'^2}X^4+Y^4+Z^4+X^2Y^2=0$ are
$\overline{\Q}$-isomorphic (recall that $C_{m,q}$ and $C_{m,q'}$ are
twists for $C_{(1-m)/4q^2}$ and $C_{(1-m)/4q'^2}$, respectively),
but this contradicts the fact that the family
$\mathcal{C}'_{A}:AX^4+Y^4+Z^4+X^2Y^2=0$ is a representative family
for over $\Q$ for the stratum
$\widetilde{\mathcal{M}_3^{\operatorname{Pl}}}(\operatorname{GAP}(16,13))$.
\end{proof}
\end{subsection}

\begin{rem} The previous discussion tends to be applicable
for any zero-dimensional stratum
$\widetilde{\mathcal{M}_3^{\operatorname{Pl}}}(G)$, that is, when
$G=\{\operatorname{GAP}(48,33),
\operatorname{GAP}(96,64),\operatorname{PSL}_2(\F_7)\}$. {However,
once we start with families that parameterize the twists over $\Q$,
we need to precise the algebraic restrictions on the parameters that
appear in the defining equations, which characterize when two twists
are $\Q$-equivalent. This is a key point in order to construct an
infinite family of non $\Q$-isomorphic smooth plane quartic curves
with infinitely (resp. finitely) many quadratic points over $\Q$. In
particular, one may check Conjecture \ref{conj} for Fermat and Klein
quartic curves considering all the details of constructing their
twists in \cite{PhDLorenzo}}.
\end{rem}

\end{section}

\begin{section}*{Acknowledgements} {The authors are grateful to Andreas Schweizer for his e-mails that improved the
part on rational and quadratic points over a global
field of positive characteristic. We also would like} to thank Luis Dieulefait, Joaquim Ro\'e and Xavier Xarles for their interesting
comments.
\end{section}

\end{document}